\documentclass[12pt]{article}
\usepackage{amsmath}
\usepackage{amsfonts}
\usepackage{amssymb}
\usepackage{hyperref}

\usepackage{graphicx}
\providecommand{\U}[1]{\protect\rule{.1in}{.1in}}
\newtheorem{theorem}{Theorem}

\newtheorem{lemma}[theorem]{Lemma}

\newenvironment{proof}[1][Proof]{\noindent\textbf{#1.} }{{\hfill $\Box$ \\}}
\begin{document}

\title{Orbits of irreducible binary forms over GF$(p)$}
\author{Michael Vaughan-Lee}
\date{May 2017}
\maketitle

\begin{abstract}
In this note I give a formula for calculating the number of orbits of
irreducible binary forms of degree $n$ over GF$(p)$ under the action of
GL$(2,p)$. This formula has applications to the classification of class two
groups of exponent $p$ with derived groups of order $p^{2}$.

\end{abstract}

\section{Introduction}

A binary form of degree $n$ over a field $F$ is a homogeneous polynomial%
\[
\alpha_{0}x^{n}+\alpha_{1}x^{n-1}y+\alpha_{2}x^{n-2}y^{2}+\ldots+\alpha
_{n}y^{n}%
\]
in $x,y$ with coefficients in $F$. Two binary forms are taken to be identical
if one is a scalar multiple of the other. We define an action of GL$(2,F)$ on
binary forms as follows. Let $g=\left(
\begin{array}
[c]{cc}%
a & b\\
c & d
\end{array}
\right)  \in\,$GL$(2,F)$, and let%
\[
f=\alpha_{0}x^{n}+\alpha_{1}x^{n-1}y+\alpha_{2}x^{n-2}y^{2}+\ldots+\alpha
_{n}y^{n}.
\]
Then we define%
\[
fg=\alpha_{0}(ax+by)^{n}+\alpha_{1}(ax+by)^{n-1}(cx+dy)+\ldots+\alpha
_{n}(cx+dy)^{n}.
\]
Binary forms over the complex numbers are an important research topic, but
there is very little published in the literature on binary forms over GF$(p)$.
However binary forms over GF$(p)$ do have applications to the classification
of class two groups of exponent $p$ with derived groups of order $p^{2}$
($p>2$). There is an important paper by Vishnevetskii \cite{vish85} in which
he classifies the indecomposable groups of this form. (A group of this form is
indecomposable if it cannot be expressed as a central product of two proper
subgroups.) Let us call a class two group $G$ of exponent $p$ with derived
group of order $p^{2}$ a $(d,2)$ group if $|G/G^{\prime}|=p^{d}$, so that $G$
has $d$ generators. Vishnevetskii shows that if $d$ is odd then there is only
one indecomposable $(d,2)$ group. If $d=3$ then it has a presentation on
generators $a_{1},a_{2},a_{3}$ with a single relation $[a_{1},a_{3}]=1$ in
addition to the relations making it a class two group of exponent $p$. If
$\ d$ is odd and $d>3$ then it has presentation on generators $a_{1}%
,a_{2},\ldots,a_{d}$ with relations%
\begin{gather*}
\lbrack a_{1},a_{2}]=[a_{3},a_{4}]=[a_{5},a_{6}]=\ldots=[a_{d-2},a_{d-1}],\\
\lbrack a_{2},a_{3}]=[a_{4},a_{5}]=\ldots=[a_{d-1},a_{d}],
\end{gather*}
where all other commutators of the generators are trivial. If $d$ is even, say
$d=2n$, then the indecomposable groups of type $(2n,2)$ correspond to orbits
of binary forms $f$ of degree $n$ which have the form $f=g^{k}$ where $g$ is
irreducible. Corresponding to an orbit representative $f$ of degree $n $ we
have an indecomposable group $V_{f}$ on generators%
\[
x_{1},x_{2},\ldots,x_{n},y_{1},y_{2},\ldots,y_{n},z_{1},z_{2}.
\]
The relations on $V_{f}$ are given by a pair of $n\times n$ Scharlau matrices
$A,B$ where $A$ is the identity matrix, and $B$ is the companion matrix of the
polynomial $f(x,1)$. We set $[x_{i},x_{j}]=[y_{i},y_{j}]=1$ for all $i,j$. Let
$A$ have $(i,j)$-entry $a_{ij}$ and let $B$ have $(i,j)$-entry $b_{ij}$. Then
we set $[x_{i},y_{j}]=z_{1}^{a_{ij}}z_{2}^{b_{ij}}$. (The generators
$z_{1},z_{2}$ are assumed to be central, of course.) The type $(2n,2)$ groups
obtained in this way form a complete and irredundant set of indecomposable
groups of type $(2n,2)$. For example, consider the case $n=4$. If $f=g^{k}$
where $g$ is irreducible, then $k=1,2$ or 4. If $k=1 $ then then $f$ is
irreducible, and there are $\frac{p+1}{2}$ orbits of irreducible quartic
binary forms over GF$(p)$. If $k=2$ then $f=g^{2}$ where $g$ is an irreducible
quadratic, and there is only one orbit of irreducible quadratics. And if $k=4$
then $f=g^{4}$ where $g$ is irreducible of degree 1, and there is only one
orbit of irreducible binary forms of degree 1. So there are $\frac{p+5}{2}$
indecomposable groups of type $(8,2)$, and if we pick representatives for the
orbits of irreducible binary forms of degrees 1,2 and 4 we can write down
presentations for the groups.

Vishnevetskii \cite{vish82} gives a formula to compute the number of orbits of
irreducible binary forms of degree $n$ over GF$(p)$ for the cases when $n$ is
coprime to $p+1$. I have managed to extend Vishnevetskii's formula to cover all
$n>2$. (When $n\leq2$ there is exactly one orbit for all $p$.) When $n=3 $
there is one orbit for all $p$. When $n=4$ there is one orbit for $p=2$ and
$\frac{p+1}{2}$ orbits for $p>2$. When $n=5$ there is one orbit for $p=2$, six
orbits for $p=5$ and $\frac{1}{5}(p^{2}-1+2\gcd(p^{2}-1,5))$ orbits for
$p\neq2,5$. In general if $p$ is an odd prime coprime to $n$ then the number
of orbits is one of $\varphi(n)$ polynomials in $p$, with the choice of
polynomial depending on $p\operatorname{mod}n$. (If $p$ is coprime to $n$,
then $p\operatorname{mod}n$ is coprime to $n$.) So the number of orbits is
polynomial on residue classes (PORC). I have a \textsc{Magma} program which
computes these $\varphi(n)$ polynomials (with symbolic $p$) for any given $n$.
I also have a \textsc{Magma} program which computes the number of orbits for
any given prime $p$ and any given $n$, including the prime 2 and primes
dividing $n$. The programs are superficially quite complicated, but their
complexity is bounded by $k^{2}$ where $k$ is the number of divisors of $n$.
The programs can be found on my website \url{http://users.ox.ac.uk/~vlee/PORC/orbitsirredpols}.

\section{Vishnevetskii's method}

Let $S$ be the set of irreducible binary forms of degree $n$ over GF$(p)$, and
let $G=\,$GL$(2,p)$. We have an action of $G$ on $S$, and the number of orbits
is given by Burnside's Lemma:%
\[
\frac{1}{|G|}\sum_{g\in G}\text{fix}(g),
\]
where fix$(g)$ is the number of elements $s\in S$ such that $sg=s$. Since we
have a group action, fix$(g)$ depends only on the conjugacy class of $g$, and
so we only need to compute fix$(g)$ for one representative $g$ from each
conjugacy class. There are four types of conjugacy class.

\begin{enumerate}
\item There are $p-1$ conjugacy classes of size one, each containing an
element of the form $g=\left(
\begin{array}
[c]{cc}%
\lambda & 0\\
0 & \lambda
\end{array}
\right)  $. For $g$ of this form%
\[
\text{fix}(g)=|S|=\sum_{d|n}\mu(d)p^{n/d},
\]
where $\mu$ is the M\"{o}bius function.

\item There are $p-1$ conjugacy classes of size $p^{2}-1$, each containing an
element $g=\left(
\begin{array}
[c]{cc}%
\lambda & \lambda\\
0 & \lambda
\end{array}
\right)  $. For each of these elements
\[
\text{fix}(g)=\,\text{fix}\left(
\begin{array}
[c]{cc}%
1 & 1\\
0 & 1
\end{array}
\right)  .
\]
We will show in Section 3 that this is given by Vishnevetskii's function
$B(p,n)$ from \cite{vish82}, where $B(p,n)=0$ if $p\nmid n$ and where%
\[
B(p,n)=\frac{p-1}{n}\sum_{d|n,\,p\nmid d}\mu(d)p^{n/pd}%
\]
if $p|n$.

\item There are $\frac{1}{2}(p-1)(p-2)$ conjugacy classes of size $p^{2}+p$
each containing an element $g=\left(
\begin{array}
[c]{cc}%
\lambda & 0\\
0 & \mu
\end{array}
\right)  $ with $\lambda\neq\mu$. We will show in Section 4 below that if
$g=\left(
\begin{array}
[c]{cc}%
\lambda & 0\\
0 & \mu
\end{array}
\right)  $ then fix$(g)$ depends only on the multiplicative order of
$\frac{\lambda}{\mu}$. We will show that if the order of $\frac{\lambda}{\mu}$
does not divide $n$ then fix$(g)=0$, and that if the order is $e$ where $e$
divides $n$ then fix$(g)$ is given by Vishnevetskii's function $A(p,n,e)$ from
\cite{vish82}. If $e|n$ then we write $\frac{n}{e}=kr$ where $k$ is the
largest possible divisor of $n$ which is coprime to $e$ and then%
\[
A(p,n,e)=\frac{\varphi(e)}{n}\sum_{d|k}\mu(d)(p^{kr/d}-1).
\]
Note that for each $e>1$ dividing $p-1$ there are $\varphi(e)\frac{p-1}{2}$
conjugacy classes containing an element $\left(
\begin{array}
[c]{cc}%
\lambda & 0\\
0 & \mu
\end{array}
\right)  $ with $\frac{\lambda}{\mu}$ of order $e$.

\item There are $\frac{1}{2}p(p-1)$ conjugacy classes of size $p^{2}-p$,
containing elements $g\in G$ whose eigenvalues do not lie in GF$(p)$. Let $N$
be the central subgroup of $G$ consisting of the matrices $\lambda I$. We will
show in Section 5 that if $g$ lies in one of these conjugacy classes then
fix$(g)$ depends only on the order of $gN$. Note that this order must divide
$p+1$, and that for every $e>1$ dividing $p+1$ there are $\varphi(e)\frac
{p-1}{2}$ conjugacy classes of this form containing elements $g$ with $gN$ of
order $e$. We will show that if $gN$ has order $e$ where $e\nmid n$, then
fix$(g)=0$. Vishnevetskii does not have an expression for fix$(g)$ when $e|n$,
and this is why his formula only applies when $n$ is coprime to $p+1$. We will
obtain a function $C(p,n,e)$ in Section 5 which gives fix$(g)$ in the case
when $e|\gcd(n,p+1)$.
\end{enumerate}

Putting all this together we see that the number of orbits is%
\[
\frac{1}{|G|}(a+b+c+d),
\]
where%
\begin{align*}
a  & =(p-1)\sum_{d|n}\mu(d)p^{n/d},\\
b  & =(p-1)B(p,n),\\
c  & =\sum_{e|(n,p-1),\,e\neq1}\varphi(e)\frac{p-1}{2}(p^{2}+p)A(p,n,e),\\
d  & =\sum_{e|(n,p+1),\,e\neq1}\varphi(e)\frac{p-1}{2}(p^{2}-p)C(p,n,e).
\end{align*}

\section{Vishnevetskii's formula $B(p,n)$}

We need to calculate fix$(g)$ when $g=\left(
\begin{array}
[c]{cc}%
1 & 1\\
0 & 1
\end{array}
\right)  $. We take $x^{n}$, $x^{n-1}y$, $x^{n-2}y^{2}$, \ldots, $y^{n}$ as a basis for
the space of homogeneous polynomials of degree $n$ in $x,y$ over GF$(p)$. Then
$g$ maps these basis elements to
\[
(x+y)^{n},\,(x+y)^{n-1}y,\,\ldots,\,y^{n},
\]
and so the matrix $A$ giving the action of $g$ is an upper triangular matrix
with 1's down the diagonal, and entries $n,n-1,\ldots,2,1$ down the
superdiagonal. The only eigenvalue is 1, and we need to count the number of
eigenvectors corresponding to irreducible polynomials. (We treat two
eigenvectors as being equal if one is a scalar multiple of the other.) An
eigenvector of $A$ corresponding to an irreducible polynomial must have
non-zero first entry, but $(1,\ast,\ast,\ldots,\ast)$ can only be an
eigenvector if the $(1,2)$-entry of $A$ is zero. So the number of eigenvectors
corresponding to irreducible polynomials is zero unless $n=0\operatorname{mod}%
p$. So suppose that $p|n$. Then the superdiagonal of $A$ has $\frac{n}{p}$
zero entries, and $\frac{(p-1)n}{p}$ non-zero entries. This implies that the
dimension of the eigenspace of $A$ is at most $1+\frac{n}{p} $. Now $y$ and
$x^{p}-xy^{p-1}$ are both fixed by $g$, and so if we let $k=\frac{n}{p}$ then
we see that the following $1+\frac{n}{p}$ polynomials are all fixed by $g$:%
\[
(x^{p}-xy^{p-1})^{k},\,(x^{p}-xy^{p-1})^{k-1}y^{p},\,\ldots,\,
(x^{p}-xy^{p-1})y^{(k-1)p},\,y^{kp}.
\]
So the eigenspace of $A$ corresponding to eigenvalue 1 has dimension
$1+\frac{n}{p}$ and these polynomials form a basis for the space of all
polynomials of degree $n$ which are fixed by the $g$. So we need to count the
number of irreducible polynomials%
\[
(x^{p}-xy^{p-1})^{k}+\alpha_{k-1}(x^{p}-xy^{p-1})^{k-1}y^{p}+\ldots+\alpha
_{1}\,(x^{p}-xy^{p-1})y^{(k-1)p}+\alpha_{0}y^{kp}%
\]
with $\alpha_{0},\alpha_{1},\ldots,a_{k-1}\in\,$GF$(p)$.

Let $F=\,$GF$(p)$, let $K=\,$GF$(p^{k})$ and let $L=\,$GF$(p^{n})$. We want to
count the number of irreducible polynomials $f(x^{p}-x)$ where $f(x)$ is a
polynomial of degree $k$. Clearly, if $f(x^{p}-x)$ is irreducible then $f(x)$
is irreducible, and so $f(x)$ splits in $K$ and $f(x^{p}-x)$ splits in $L$.
Let $\alpha$ be a root of $f(x)$ in $K$. Then $x^{p}-x-\alpha\in K[x]$ divides
$f(x^{p}-x)$, and so $x^{p}-x-\alpha$ splits in $L$. We then have
$K=F(\alpha)$, $L=F(\beta)$ if $\beta$ is any root of $x^{p}-x-\alpha$. So
$f(x)$ is the minimum polynomial over $F$ of an element $\alpha\in K$ where
$\alpha$ does not lie in any proper subfield of $K$ and $\alpha\neq\gamma
^{p}-\gamma$ for any $\gamma\in K$.

Conversely let $\alpha\in K$, and assume that $\alpha$ does not lie in any
proper subfield of $K$. Also assume that $\alpha\neq\gamma^{p}-\gamma$ for any
$\gamma\in K$. Let $f(x)$ be the minimum polynomial of $\alpha$ over $F$.
Since $\alpha$ does not lie in any proper subfield of $K$, we have
$K=F(\alpha)$, and so $f$ has degree $k$. Also since $\alpha\neq\gamma
^{p}-\gamma$ for any $\gamma\in K$ we see that $x^{p}-x-\alpha\in K[x]$ does
not split in $K$. Let $\beta$ be a root of $x^{p}-x-\alpha$ in a splitting
field $M$ for $x^{p}-x-\alpha$. The the other $p-1$ roots are $1+\beta
,2+\beta,\ldots,p-1+\beta$. Since $\beta\notin K$ the minimum polynomial of
$\beta$ over $K$ has a root $i+\beta$ with $i\neq0$, and so there is an
automorphism $\theta$ in the Galois group of $M$ over $K$ such that
$\theta(\beta)=i+\beta$. But then $\theta$ has order $p$ so that $|M:K|=p$ and
$|M:F|=p^{n}$. So the minimum polynomial of $\beta$ over $F$ has degree $n$.
But $f(x^{p}-x)$ has a root $\beta$ and has degree $n$, so must be irreducible.

So we need to count the number of elements $\alpha\in K$ which do not lie in
any proper subfield of $K$ and which cannot be expressed in the form
$\gamma^{p}-\gamma$ with $\gamma\in K$.

The map $\varphi:K\rightarrow K$ given by $\varphi(\gamma)=\gamma^{p}-\gamma$
is an additive homomorphism of $K$ with kernel $F$. So $|\operatorname{Im}%
\varphi|=p^{k-1}$, and the set $S=K\backslash\operatorname{Im}\varphi$
contains $\frac{p-1}{p}p^{k}$ elements. We need to count the number of
elements of $S$ which do not lie in a proper subfield of $K$. So let $M$ be a
proper subfield of $K$. We show that if $p$ divides $|K:M|$ then $M\cap
S=\varnothing$, and that if $|K:M|$ is coprime to $p$ then $|M\cap
S|=\frac{p-1}{p}|M|$.

First consider the case when $p$ divides $|K:M|$. If $\alpha\in M$ is not
equal to $\gamma^{p}-\gamma$ for any $\gamma\in M$ then, as we saw above, the
splitting field of $x^{p}-x-\alpha$ over $M$ has degree $p$ over $M$ and so
must be contained in $K$. So $M\cap S=\varnothing$.

Now consider the case when $|K:M|$ is coprime to $p$. We show that if
$\beta\in K\backslash M$ then $\beta^{p}-\beta\in K\backslash M$. For suppose
that $\beta\in K\backslash M$ and $\beta^{p}-\beta=\alpha\in M$. Then the
splitting field of $x^{p}-x-\alpha$ over $M$ has degree $p$ over $M$ and is
contained in $K$, which is impossible. It follows that $\varphi(K\backslash
M)\subset K\backslash M$ so that $|S\cap(K\backslash M)|=\frac{p-1}%
{p}|K\backslash M|$ and $|S\cap M|=\frac{p-1}{p}|M|$.

Putting all this together we see that the number of elements of $S$ which do
not lie in any proper subfield of $K$ is%
\[
\frac{p-1}{p}\sum_{d}\mu(d)p^{n/pd}%
\]
where the sum is taken over all $d$ dividing $k$ which are coprime to $p$. It
follows from this that the number of polynomials $f(x)$ of degree $k$ with
$f(x^{p}-x)$ irreducible is%
\[
\frac{p-1}{n}\sum_{d}\mu(d)p^{n/pd},
\]
and this is Vishnevetskii's function $B(p,n)$.

\section{Vishnevetskii's function $A(p,n,e)$}

Let $g=\left(
\begin{array}
[c]{cc}%
\lambda & 0\\
0 & \mu
\end{array}
\right)  $ with $\lambda\neq\mu$. Then fix$(g)=\,$fix$(h)$ where $h=\left(
\begin{array}
[c]{cc}%
\nu & 0\\
0 & 1
\end{array}
\right)  $, with $\nu=\frac{\lambda}{\mu}$. We take $x^{n},x^{n-1}%
y,x^{n-2}y^{2},\ldots,y^{n}$ as a basis for the space of homogeneous
polynomials of degree $n$ in $x,y$ over GF$(p)$. Then $h$ maps these basis
elements to%
\[
\nu^{n}x^{n},\,\nu^{n-1}x^{n-1}y,\,\nu^{n-2}x^{n-2}y^{2},\,\ldots,\,y^{n}.
\]
So the matrix $M$ giving the action of $h$ is a diagonal matrix with entries
$\nu^{n},\nu^{n-1},\ldots,1$ down the diagonal. We need to count eigenvectors
of $M$ corresponding to irreducible polynomials, and such an eigenvector must
have non-zero first entry and non-zero last entry. So fix$(h)=0$ unless
$\nu^{n}=1$. So suppose that $\nu$ has order $e$ where $e|n$. Then the
eigenspace corresponding to eigenvalue 1 is spanned by $x^{n}$, $x^{n-e}y^{e}%
$, $x^{n-2e}y^{2e}$, \ldots, $y^{n}$, and so if we set $\frac{n}{e}=kr$ where
$k$ is the largest possible divisor of $n$ which is coprime to $e$, we need to
count irreducible polynomials of the form%
\[
x^{n}+\alpha_{1}x^{n-e}y^{e}+\alpha_{2}x^{n-2e}y^{2e}+\ldots+\alpha_{kr}y^{n}.
\]
This is equivalent to counting the number of irreducible polynomials
$f(x^{e})$ where $f(x)$ is a polynomial of degree $kr$.

Let $F=\,$GF$(p)$, $K=\,$GF$(p^{kr})$ and let $L=\,$GF$(p^{n})$. Let $f(x)\in
F[x]$ have degree $kr$ and suppose that $f(x^{e})$ is irreducible. Then $L$ is
the splitting field of $f(x^{e})$. Also, $f(x)$ must be irreducible over $F$,
so that $K$ is the splitting field of $f(x)$. Let $\alpha\in K$ be a root of
$f(x)$, so that $K=F(\alpha)$, and $\alpha$ does not lie in any proper
subfield of $K$. Then $x^{e}-\alpha$ divides $f(x^{e})$ over $K$, and so
$x^{e}-\alpha$ has a root $\beta\in L$, and $L=K(\beta)$. Since $|L:K|=e$,
$x^{e}-\alpha$ must be irreducible over $K$, and so $\alpha$ must be non-zero
and cannot be the $q^{th}$ power of an element of $K$ for any prime $q$
dividing $e$.

Conversely, suppose that $\alpha$ is a non-zero element of $K$ which does not
lie in any proper subfield of $K$, and suppose that $\alpha$ is not equal to a
$q^{th}$ power of an element of $K$ for any prime $q$ dividing $e$. Let $f(x)$
be the minimum polynomial of $\alpha$ over $F$. Since we must have
$K=F(\alpha)$, $f(x)$ is an irreducible polynomial of degree $kr$. We show
that $f(x^{e})$ is irreducible over $F$.

Let $M$ be the splitting field of $x^{e}-\alpha$ over $K$, and let $\gamma\in
M$ be a root of $x^{e}-\alpha$. Since $p=1\operatorname{mod}e$, there is a
primitive $e^{th}$ root of unity $\zeta\in F$, and the roots of $x^{e}-\alpha$
in $M$ are%
\[
\gamma,\zeta\gamma,\zeta^{2}\gamma,\ldots,\zeta^{e-1}\gamma.
\]
So the conjugates of $\gamma$ over $K$ have the form $\zeta^{i}\gamma$, and
the set of powers $\zeta^{i}$ such that $\zeta^{i}\gamma$ is conjugate to
$\gamma$ form a subgroup of $\langle\zeta\rangle$. Let this subgroup have
order $m$ dividing $e$. Then the minimum polynomial of $\gamma$ over $K$ is
$x^{m}-\gamma^{m}$ (with $\gamma^{m}\in K$). So%
\[
\alpha=\gamma^{e}=(\gamma^{m})^{e/m}.
\]
Our assumption that $\alpha$ is not equal to a $q^{th}$ power of an element of
$K$ for any prime $q$ dividing $e$ implies that $m=e$. So $|M:K|=e$ and
$M=L=F(\gamma)$. This implies that the minimum polynomial of $\gamma$ over $F$
has degree $n$. But $f(x^{e})$ is a polynomial of degree $n$ which has
$\gamma$ as a root, and so $f(x^{e})$ must be the minimum polynomial of
$\gamma$ over $F$, and must be irreducible.

So to count the number of irreducible polynomials of the form $f(x^{e})$ where
$f$ has degree $kr$ we need to count the number of elements $\alpha\in K$
which do not lie in any proper subfield of $K$ and are not $q^{th}$ powers of
elements of $K$ for any prime $q$ dividing $e$. Let $S$ be the set of elements
in $K$ which are not $q^{th}$ powers of elements in $K$ for any prime $q$
dividing $e$. The non-zero elements of $K$ form a cyclic group $G$ of order
$p^{kr}-1$, and $e|p-1$, so $|S|=\frac{\varphi(e)}{e}(p^{kr}-1)$. However we
need to take account of elements of $S$ which lie inside proper subfields of
$K$. So let $M$ be a proper subfield of $K$, with $|K:M|=t$.

First consider the case when $q|t$ for some prime $q$ dividing $e$. Then if
$\alpha\in M$ is not a $q^{th}$ power of an element in $M$ then $x^{q}-\alpha$
is irreducible over $M$. (This uses the fact that $F$ contains primitive
$q^{th}$ roots of unity.) So the splitting field of $x^{q}-\alpha$ over $M$
has degree $q$ over $M$, and must be contained in $K$ since $q$ divides
$|K:M|$. So $|S\cap M|=0$.

Next consider the case when $t$ is coprime to $e$. Note that this implies that
$t|k$. We show that in this case $|S\cap M|=\frac{\varphi(e)}{e}(|M|-1) $. To
see this suppose that $\alpha\in M\backslash S$. Then $\alpha=\gamma^{q}$ for
some $\gamma\in K$ and some prime $q|e$. If $\gamma\notin M$ then
$x^{q}-\alpha$ is irreducible over $M$, so that the splitting field of
$x^{q}-\alpha$ over $M$ is an extension of $M$ of degree $q$. But this is
impossible since this splitting field is $M(\gamma)$ which is a subfield of
$K$, and $q$ does not divide $|K:M|$. So if $\alpha\in M\backslash S$ then
$\alpha$ is a $q^{th}$ power of some element of $M$ for some $q|e$. Hence
$|S\cap M|=\frac{\varphi(e)}{e}(|M|-1)$.

It follows from this that the number of elements of $S$ which do not lie in
proper subfields of $K$ is%
\[
\frac{\varphi(e)}{e}\sum_{d|k}\mu(d)(p^{kr/d}-1).
\]

Each of these elements has a minimal polynomial $f(x)$ of degree $kr$ with
$f(x^{e})$ irreducible, and so the number of these polynomials is%
\[
\frac{\varphi(e)}{n}\sum_{d|k}\mu(d)(p^{kr/d}-1).
\]

\section{My function $C(p,n,e)$}

If $e|\gcd(n,p+1)$ and $e>1$ then we write $\frac{n}{e}=kr$ where $k$ is the
largest possible divisor of $n$ which is coprime to $e$. We define%
\[
C(p,n,e)=\frac{\varphi(e)}{n}\sum_{d|k}\mu(\frac{k}{d})(p^{rd}%
-1+2(rd\operatorname{mod}2)).
\]
We show that if $g\in G$ has eigenvalues that do not lie in GF$(p)$, and if
$gN$ has order $e$ then fix$(g)=0$ if $e\nmid n$, and fix$(g)=C(p,n,e)$ if
$e|n$.

Let $g\in G$ have eigenvalues that do not lie in GF$(p)$. Then $g$ is
conjugate to an element $h=\left(
\begin{array}
[c]{cc}%
0 & r\\
1 & s
\end{array}
\right)  $ where $x^{2}-sx-r$ is irreducible over GF$(p)$, and fix$(g)=\,$%
fix$(h)$. So we assume that $g=\left(
\begin{array}
[c]{cc}%
0 & r\\
1 & s
\end{array}
\right)  $. Let $\lambda,\lambda^{p}$ be the eigenvalues of $g$ in GF$(p^{2}%
)$. Then $g$ has eigenvectors $(1,\lambda)$, $(1,\lambda^{p})$ with these
eigenvalues. Let $U$ be the space of homogeneous polynomials of degree $n$ in
$x,y$ over GF$(p^{2})$, and let $V$ be the space of homogeneous polynomials of
degree $n$ in $x,y$ over GF$(p)$. Then $U$ has a basis%
\[
(x+\lambda y)^{n},(x+\lambda y)^{n-1}(x+\lambda^{p}y),(x+\lambda
y)^{n-2}(x+\lambda^{p}y)^{2},\ldots,(x+\lambda^{p}y)^{n}.
\]
The element $g$ acts as a linear transformation $T_{g}$ on $U$, and these
basis vectors are eigenvectors for $T_{g}$ with eigenvalues $\lambda^{n}$,
$\lambda^{n-1}\lambda^{p}$, \ldots, $\lambda^{pn}$. We show that fix$(g)=0 $
unless $\lambda^{n}=\lambda^{pn}$. (Note that if $\lambda^{n}=\lambda^{pn}$
then $g^{n}=\left(
\begin{array}
[c]{cc}%
\lambda^{n} & 0\\
0 & \lambda^{n}%
\end{array}
\right)  \in N$, and the order of $gN$ divides $n$.)

The element $g$ acts as a linear transformation $S_{g}$ on $V$, and $v\in V$
is fixed by $g$ if and only if $v$ is an eigenvector for $S_{g}$. So let $v\in
V$ and suppose that $vS_{g}=\mu v$ with $\mu\in\,$GF$(p)$. The eigenvalues of
$S_{g}$ are also eigenvalues of $T_{g}$, and so%
\[
\mu\in\{\lambda^{n},\lambda^{n-1}\lambda^{p},\ldots,\lambda^{np}\}.
\]
Now $\mu=\mu^{p}$, and so if $\lambda^{n}\neq\lambda^{np}$ then%
\[
\mu\in\{\lambda^{n},\lambda^{n-1}\lambda^{p},\ldots,\lambda^{np}%
\}\backslash\{\lambda^{n},\lambda^{np}\}.
\]
But this implies that $v$ lies in the GF$(p^{2})$ span of%
\[
(x+\lambda y)^{n-1}(x+\lambda^{p}y),(x+\lambda y)^{n-2}(x+\lambda^{p}%
y)^{2},\ldots,(x+\lambda y)(x+\lambda^{p}y)^{n-1},
\]
and hence that $v$ has a factor%
\[
(x+\lambda y)(x+\lambda^{p}y)=x^{2}+sxy-ry^2
\]
and is not irreducible. (We are assuming that $n>2$.) So if $\lambda^{n}%
\neq\lambda^{np}$ then fix$(g)=0$, as claimed.

So assume that $\lambda^{n}=\lambda^{np}$ and let $gN$ have order $e$ dividing
$n$. Then $e|p+1$, and $e$ is the smallest value of $s$ such that
$\lambda^{s}\in\,$GF$(p)$. As we have seen, if $v$ is an irreducible
polynomial in $V$ then $v$ is an eigenvector for $S_{g}$ with eigenvalue
$\lambda^{n}$. The dimension of this eigenspace is $1+\frac{n}{e}$ and we
obtain a basis for the eigenspace as follows. As in the definition of
$C(n,p,e)$ let $\frac{n}{e}=kr$ where $k$ is the largest divisor of $n$ which
is coprime to $e$. Let%
\begin{align*}
a(x,y)  & =\frac{1}{2}\left(  (x+\lambda y)^{e}+(x+\lambda^{p}y)^{e}\right)
,\\
b(x,y)  & =\frac{1}{e(\lambda-\lambda^{p})}\left(  (x+\lambda y)^{e}%
-(x+\lambda^{p}y)^{e}\right)  .
\end{align*}
Then $a$ and $b$ are homogeneous polynomials of degree $e$ in GF$(p)[x,y]$.
The coefficient of $x^{e}$ in $a$ is 1, and the coefficient of $y^{e}$ in $a$
is $\lambda^{e}$. The coefficients of $x^{e}$ and $y^{e}$ in $b$ are zero, and
the coefficient of $x^{e-1}y$ is 1. The eigenspace of $S_{g}$ for eigenvector
$\lambda^{n}$ has basis%
\[
a^{kr},a^{kr-1}b,a^{kr-2}b^{2},\ldots,b^{kr}.
\]
So to calculate fix$(g)$ we need to count the number of irreducible
polynomials of the form%
\[
\alpha_{0}a^{kr}+\alpha_{1}a^{kr-1}b+\alpha_{2}a^{kr-2}b^{2}+\ldots
+\alpha_{kr}b^{kr}.
\]
Since $b$ is divisible by $x$ we can assume that $\alpha_{0}=1$.

So assume that%
\[
h(x,y)=a^{kr}+\alpha_{1}a^{kr-1}b+\alpha_{2}a^{kr-2}b^{2}+\ldots+\alpha
_{kr}b^{kr}%
\]
is irreducible, and let%
\[
f(x)=x^{kr}+\alpha_{1}x^{kr-1}+\alpha_{2}x^{kr-2}+\ldots+\alpha_{kr}.
\]
Then $f$ must be irreducible, and so $f$ has splitting field $K=\,$%
GF$(p^{kr})$. Let $F=\,$GF$(p)$, and let $L=\,$GF$(p^{n})$. Since $h(x,1)$ is
irreducible over $F$ it splits over $L$. Let $\beta$ be a root of $f(x)$ in
$K$. Then $a(x,1)-\beta b(x,1)\in K[x]$ divides $h(x,1)$, and so splits in $L
$. Let $\gamma$ be a root of $a(x,1)-\beta b(x,1)$ in $L$, so that
$L=F(\gamma)$. So $|K(\gamma):K|=|L:K|=e$, which implies that $a(x,1)-\beta
b(x,1)$ is irreducible over $K$.

Conversely, suppose that $\beta\in K$ generates $K$ over $F$, and suppose that
$a(x,1)-\beta b(x,1)$ is irreducible over $K$. Let%
\[
f(x)=x^{kr}+\alpha_{1}x^{kr-1}+\alpha_{2}x^{kr-2}+\ldots+\alpha_{kr},
\]
be the minimum polynomial of $\beta$, and let%
\[
h(x,y)=a^{kr}+\alpha_{1}a^{kr-1}b+\alpha_{2}a^{kr-2}b^{2}+\ldots+\alpha
_{kr}b^{kr}.
\]
Since $a(x,1)-\beta b(x,1)$ is irreducible over $K$ its splitting field over
$K$ is $L$. Let $\gamma$ be a root of $a(x,1)-\beta b(x,1)$ in $L$, so that
$L=K(\gamma)=F(\beta,\gamma)$. Provided $b(\gamma,1)\neq0$ we see that
$\beta\in F(\gamma)$, so that $L=F(\gamma)$ which implies that the minimum
polynomial of $\gamma$ over $F$ has degree $n$. This implies that $h(x,1)$ is
the minimum polynomial of $\gamma$ over $F$ so that $h(x,1)$ and $h(x,y)$ are
irreducible. However we cannot have $b(\gamma,1)=0$ since this would imply
that $a(\gamma,1)=b(\gamma,1)=0$ and this would imply that $a(x,1)$ and
$b(x,1)$ have a common factor over $F$ so that $a(x,1)-\beta b(x,1)$ would not
be irreducible over $K$.

So to count irreducible polynomials of the form $h(x,y)$ we need to count
elements $\beta\in K$ such that $\beta$ generates $K$ over $F$ and such that
$a(x,1)-\beta b(x,1)$ is irreducible over $K$. As a step towards this we prove
the following lemma.

\begin{lemma}
Let $M=\,$GF$(p^{m})$. If $m$ is even then the number of elements $\beta\in M$
such that $a(x,1)-\beta b(x,1)$ is irreducible over $M$ is $\frac{\varphi
(e)}{e}(p^{m}-1)$, and if $m$ is odd then the number of elements $\beta\in M $
such that $a(x,1)-\beta b(x,1)$ is irreducible over $M$ is $\frac{\varphi
(e)}{e}(p^{m}+1)$.
\end{lemma}

\begin{proof}
First suppose that $m$ is even. Then $\lambda$ and $\lambda^{p}$ lie in $M$,
and the set of $M$-linear combinations of $a(x,1)$ and $b(x,1)$ is the same as
the set of $M$-linear combinations of $(x+\lambda)^{e}$ and $(x+\lambda
^{p})^{e}$. So the number of irreducible polynomials of the form $a(x,1)-\beta
b(x,1)$ with $\beta\in M$ is the same as the number of irreducible polynomials
of the form $(x+\lambda)^{e}-\beta(x+\lambda^{p})^{e}$ with $\beta\in M$.
Since $e|p+1$, $M$ contains the $e^{th}$ roots of unity, and so $(x+\lambda
)^{e}-\beta(x+\lambda^{p})^{e}$ is reducible if and only if $\beta$ is a
$q^{th}$ power of some element of $M$ for some prime $q$ dividing $e$. So the
number of irreducible polynomials $(x+\lambda)^{e}-\beta(x+\lambda^{p})^{e}$
is%
\[
\sum_{d|e}\mu(d)\frac{p^{m}-1}{d}=\frac{\varphi(e)}{e}(p^{m}-1).
\]

Next, suppose that $m$ is odd, so that $\lambda,\lambda^{p}\notin M$. Let
$L=\,$GF$(p^{2m})$. So $L$ is an extension field of $M$ with $|L:M|=2$. The
field $L$ contains $\lambda,\lambda^{p}$, and also contains the $e^{th}$
roots of unity. Since $m$ is odd, $\lambda^{p}=\lambda^{p^{m}}$. If $\beta\in
M$ then%
\[
a(x,1)-\beta b(x,1)=(\frac{1}{2}+\frac{\beta}{e(\lambda-\lambda^{p^{m}}%
)})(x+\lambda)^{e}+(\frac{1}{2}-\frac{\beta}{e(\lambda-\lambda^{p^{m}}%
)})(x+\lambda^{p^{m}})^{e}.
\]
If we set $\gamma=\frac{1}{2}+\frac{\beta}{e(\lambda-\lambda^{p^{m}})}\in L$
then%
\[
a(x,1)-\beta b(x,1)=\gamma(x+\lambda)^{e}+\gamma^{p^{m}}(x+\lambda^{p^{m}%
})^{e}.
\]
Conversely, if $\gamma\in L$ then $\gamma(x+\lambda)^{e}+\gamma^{p^{m}%
}(x+\lambda^{p^{m}})^{e}$ is an $M$-linear combination of $a(x,1)$ and
$b(x,1)$. Note that $\gamma(x+\lambda)^{e}+\gamma^{p^{m}}(x+\lambda^{p^{m}%
})^{e}$ is a scalar multiple of $\delta(x+\lambda)^{e}+\delta^{p^{m}%
}(x+\lambda^{p^{m}})^{e}$ if and only if $\gamma^{p^{m}-1}=\delta^{p^{m}-1}$.
Let $\omega$ be a primitive element in $L$. Then $-1=\omega^{(p^{m}%
-1)(p^{m}+1)/2}$. We let $\zeta=\omega^{(p^{m}+1)/2}$ so that $\zeta^{p^{m}%
-1}=-1$. So%
\[
-\gamma^{p^{m}-1}=(\zeta\gamma)^{p^{m}-1}=\omega^{(p^{m}-1)c}%
\]
for some $c$ with $1\leq c\leq p^{m}+1$. Note that this implies that there are
$p^{m}+1$ different values of $\gamma^{p^{m}-1}$. We show that $\gamma
(x+\lambda)^{e}+\gamma^{p^{m}}(x+\lambda^{p^{m}})^{e}$ is irreducible over $M$
if and only if $c$ is coprime to $e$. Since $e|p^{m}+1$ this implies that
there are $\frac{\varphi(e)}{e}(p^{m}+1)$ different values of $\gamma
^{p^{m}-1}$ yielding polynomials $\gamma(x+\lambda)^{e}+\gamma^{p^{m}%
}(x+\lambda^{p^{m}})^{e}$ which are irreducible over $M$. This in turn implies
that there are $\frac{\varphi(e)}{e}(p^{m}+1)$ values of $\beta\in M$ such
that $a(x,1)-\beta b(x,1)$ is irreducible over $M$, as claimed.

So suppose that $c$ is not coprime to $e$. Then there is some prime $q$
dividing $e$ which also divides $c$. Write $c=qd$, and let $\delta=\omega^{d}%
$. Then $-\gamma^{p^{m}-1}=\delta^{(p^{m}-1)q}$. So $(x+\lambda)^{e}%
+\gamma^{p^{m}-1}(x+\lambda^{p^{m}})^{e}$ is divisible by
\[
(x+\lambda)^{e/q}-\delta^{p^{m}-1}(x+\lambda^{p^{m}})^{e/q}=(x+\lambda
)^{e/q}+(\zeta\delta)^{p^{m}-1}(x+\lambda^{p^{m}})^{e/q}%
\]
and this implies that $\gamma(x+\lambda)^{e}+\gamma^{p^{m}}(x+\lambda^{p^{m}%
})^{e}$ is divisible by
\[
\zeta\delta(x+\lambda)^{e/q}+(\zeta\delta)^{p^{m}}(x+\lambda^{p^{m}})^{e/q}.
\]

Conversely assume that $\gamma(x+\lambda)^{e}+\gamma^{p^{m}}(x+\lambda^{p^{m}%
})^{e}$ is not irreducible over $M$. Then $(x+\lambda)^{e}+\gamma^{p^{m}%
-1}(x+\lambda^{p^{m}})^{e}$ is not irreducible over $L$ and so (since $L$
contains the $e^{th}$ roots of unity)
$(x+\lambda)^{e}+\gamma^{p^{m}-1}(x+\lambda^{p^{m}})^{e}$
has a factorization
\[
\left(  (x+\lambda)^{e/t}-\alpha_{1}(x+\lambda^{p^{m}})^{e/t}\right)
\ldots\left(  (x+\lambda)^{e/t}-\alpha_{t}(x+\lambda^{p^{m}})^{e/t}\right)
\]
for some $t>1$, with $\alpha_{1}^{t}=\alpha_{2}^{t}=\ldots=\alpha_{t}%
^{t}=-\gamma^{p^{m}-1}$.

First consider the case when $t$ is divisible by an odd prime $q$. Then
$-\gamma^{p^{m}-1}=\omega^{(p^{m}-1)c}$ is a $q^{th}$ power, and since $q$ is
coprime to $p^{m}-1$ but divides $p^{m}+1$ this implies that $q|c$.

Next consider the case when $t$ is divisible by 4. Note that since $t|e$ and
$e|p+1$ this implies that $\frac{p^{m}-1}{2}$ is odd, and also implies that
$2$ is a prime dividing $e$. Now $-\gamma^{p^{m}-1}=\omega^{(p^{m}-1)c}$ is a
fourth power of an element of $L$, and since $\frac{p^{m}-1}{2}$ is odd, this
implies $2|c$.

Finally suppose that $(x+\lambda)^{e}+\gamma^{p^{m}-1}(x+\lambda^{p^{m}})^{e}$
has no factorization with $t$ divisible by an odd prime or $t$ divisible
by 4. Since we are assuming that $(x+\lambda)^{e}+\gamma^{p^{m}-1}%
(x+\lambda^{p^{m}})^{e}$ is not irreducible the only possibility left is that
$e$ is even and that
$(x+\lambda)^{e}+\gamma^{p^{m}-1}(x+\lambda^{p^{m}})^{e}$
equals
\[
\left(  (x+\lambda)^{e/2}-\alpha(x+\lambda^{p^{m}})^{e/2}\right)  
\left(  (x+\lambda)^{e/2}+\alpha(x+\lambda^{p^{m}})^{e/2}\right)
\]
for some $\alpha$ with $-\alpha^{2}=\gamma^{p^{m}-1}$. Since there is no
factorization with $t>2$ the factors $(x+\lambda)^{e/2}\pm\alpha
(x+\lambda^{p^{m}})^{e/2}$ are irreducible. So $\gamma(x+\lambda)^{e}%
+\gamma^{p^{m}}(x+\lambda^{p^{m}})^{e}$ can only factorize over $M$ if
$\alpha=\delta^{p^{m}-1}$ for some $\delta\in L$. But then%
\[
\delta^{(p^{m}-1)2}=\alpha^{2}=-\gamma^{p^{m}-1}=\omega^{(p^{m}-1)c}%
\]
so that $2|c$.

This completes the proof of Lemma 1.
\end{proof}

So let $g=\left(
\begin{array}
[c]{cc}%
0 & r\\
1 & s
\end{array}
\right)  $ where $x^{2}-sx-r$ is irreducible over GF$(p)$, let $g$ have
eigenvalues $\lambda,\lambda^{p}$, and suppose that $gN$ has order $e|p+1$. As
we have seen, fix$(g)=0$ unless $e|n$. So assume that $e|n$ and write
$\frac{n}{e}=kr$ where $k$ is the largest divisor of $n$ which is coprime to
$e$. Let $F=\,$GF$(p)$, $K=\,$GF$(p^{kr})$, and let $L=\,$GF$(p^{n})$. As we
have seen, to calculate fix$(g)$ we need to count the number of elements
$\beta\in K$ such that $K=F(\beta)$, and such that $a(x,1)-\beta b(x,1)$ is
irreducible over $K$. Let $S$ be the set of elements $\beta\in K$ such that
$a(x,1)-\beta b(x,1)$ is irreducible over $K$. By Lemma 1,
\[
|S|=\frac{\varphi(e)}{e}(p^{kr}-1+2(kr\operatorname{mod}2)),
\]
but we need to subtract away the number of elements of $S$ which lie in proper
subfields of $K$.

So let $M$ be a proper subfield of $K$ and suppose that $|K:M|=t>1$.

First we show that if $t$ is not coprime to $e$ then $|S\cap M|=0$. So suppose
that $t$ is not coprime to $e$ but that $\beta\in S\cap M$. Then $a(x,1)-\beta
b(x,1)$ is irreducible over $M$, and so the splitting field $P$ of
$a(x,1)-\beta b(x,1)$ over $M$ has degree $e$ over $M$. Let $s=\gcd(e,t)$.
Then there is a subfield $Q$ with $M<Q\leq K$ such that $|Q:M|=s$ and $M<Q\leq
P$ with $|P:Q|=\frac{e}{s}$. But this implies that the splitting field of
$a(x,1)-\beta b(x,1)$ over $K$ has degree $\frac{e}{s}$ over $K$, so that
$a(x,1)-\beta b(x,1)$ is not irreducible over $K$.

Next consider the case when $t$ is coprime to $e$. We show that if $\beta\in
M$ and if $a(x,1)-\beta b(x,1)$ is irreducible over $M$ then $\beta\in S $, so
that%
\[
|M\cap S|=\frac{\varphi(e)}{e}(p^{kr/t}-1+2(\frac{kr}{t}\operatorname{mod}2)).
\]
So suppose that $\beta\in M$ and that $a(x,1)-\beta b(x,1)$ is irreducible
over $M$. Let $P\leq L$ be a splitting field for $a(x,1)-\beta b(x,1)$ over
$M$, and let $\alpha$ be a root of $a(x,1)-\beta b(x,1)$ in $P$. So
$P=M(\alpha)$. Since $|K:M|$ and $|P:M|$ are coprime with $|K:M|.|P:M|=|L:M|$,
$L=K(\alpha)$, and so $\beta\in S$.

So we see that the number of elements of $S$ which do not lie in any proper
subfield of $K$ is%
\[
\frac{\varphi(e)}{e}\sum_{d|k}\mu(\frac{k}{d})(p^{rd}-1+2(rd\operatorname{mod}%
2)),
\]
and hence that%
\[
\text{fix}(g)=C(p,n,e).
\]

\end{document}